\documentclass[a4paper,10pt]{amsart}
\usepackage[utf8]{inputenc}
\usepackage{indentfirst, amsfonts, amsmath, amsthm, amssymb, amscd}
\usepackage{amsmath,amsfonts,amscd,bezier}
\usepackage{graphicx}
\usepackage{color}

\usepackage{hyperref}
\usepackage[normalem]{ulem}

\usepackage{thmtools}
\usepackage{thm-restate}

\newtheorem{maintheorem}{Theorem}

\newtheorem{theorem}{Theorem}[section]

\newtheorem{corollary}[theorem]{Corollary}
\newtheorem{proposition}[theorem]{Proposition}
\newtheorem{lemma}[theorem]{Lemma}
\newtheorem{definition}[theorem]{Definition}

\theoremstyle{remark}
\newtheorem{remark}[theorem]{Remark}

\title[Measure Rigidity and Disintegration]{Measure Rigidity and Disintegration: \\ Time-one map of flows.}
\author{Gabriel Ponce}
\address{Departamento de Matem\'atica, Estat\'istica e Computa\c c\~ao Cient\'ifica,
  IMECC-UNICAMP, Campinas-SP, Brazil.}
  \email{gaponce@ime.unicamp.br}
\author{R\'egis Var\~{a}o} 
\address{Departamento de Matem\'atica, Estat\'istica e Computa\c c\~ao Cient\'ifica,
  IMECC-UNICAMP, Campinas-SP, Brazil.}
\email{regisvarao@ime.unicamp.br}

\begin{document}

\maketitle

\begin{abstract}
An invariant measure for a flow is, of course, an invariant measure for any of its time-t maps. But the converse is far from being true. Hence, one may naturally ask: What is the obstruction for an invariant measure for the time-one map to be invariant for the flow itself? We give an answer in terms of measure disintegration. Surprisingly all it takes is the measure not to be ``too much pathological in the orbits". We prove the following rigidity result. If $\mu$ is an ergodic probability for the time-one map of a flow, then it is either highly pathological in the orbits, or it is highly regular (i.e invariant for the flow). In particular this measure rigidity result is also true for measurable flows by the classical Ambrose-Kakutani's representation theorem for measurable flows.
\end{abstract}

\setcounter{tocdepth}{1}
\tableofcontents

\section{Introduction}

A basic question answered on an introductory ergodic theory course is that we may always find invariant probability measures for many dynamical systems. Two particularly large classes of such dynamical systems are homeomorphisms and continuous flows on a compact manifold. We say that a measurable map $h:X \rightarrow X$ preserves a probability $\mu$ if for every measurable set $A \subset X$, then $\mu(h^{-1}(A))=\mu(A)$. If $\phi:\mathbb R \times X \rightarrow X$ is a flow, we say that the flow $\phi$ preserves the measure $\mu$ if it is preserved for every time $t$-map $\phi_t:=\phi(t,.)$. We say that a measure $\mu$ is ergodic for a certain dynamical system if any invariant set has either full measure or zero measure. It is not at all expected that an ergodic probability for the time-$t$ map to be ergodic (in particular invariant) for the flow itself. Hence a natural question arises:

\textbf{Question:} \textit{What is the obstruction for an ergodic measure for the time-one map to be ergodic (in particular invariant) for the flow itself?} 

To the best of our knowledge, even though this seems to be a natural question it has not been treated in the literature. We are able to give a precise answer to this question in terms of measure disintegration. Surprisingly all it takes is the measure not to be ``too much pathological in the orbits". That is, we prove a measure rigidity result. If $\mu$ is an ergodic probability for the time-one map of a flow, then it is either highly pathological in the orbits, or it is highly regular (i.e invariant for the flow). This is our main result:

\begin{maintheorem}\label{theo:continuous.flow}
Let $X$ be a separable metric space and $\phi:\mathbb R \times X \rightarrow X$ be a continuous flow. Denote by $\mathcal F=\{\mathcal F(x)\}_{x\in X}$ the foliation of $X$ by orbits of the flow $\phi$. Given any ergodic Borel measure $\mu$ for the time$-1$ map $\phi_1:= \phi(1,\cdot):X\rightarrow X$, then 
\begin{enumerate}
\item  either there is a set $A \subset X$ of full $\mu$-measure such that $\mathcal F(x) \cap A$ is a discrete set of $\mathcal F(x)$. 
Moreover, there is a natural number $k\geq 1$ such that $\mathcal F(x) \cap A$ is the $\phi_1-$orbit of exactly $k$ points; or
\item for $\mu$-almost every $x\in X$ there is a measure $\mu_{\mathcal F(x)}$ on $\mathcal F(x)$ such that
\[ \mu_{\mathcal F(x)}(\phi([0,t] \times \{x\})) = 2^{-1}t, \]
as long as $\tau \mapsto \phi(\tau,x)$ is injective on $[0,t]$ and where $\mu_{\mathcal F(x)}$ is a $\phi_1$-invariant measure which normalized and restricted on a foliated chart of the orbit is a disintegration of $\mu$. In particular $\mu$ is invariant for the flow.
\end{enumerate}
\end{maintheorem}

 In \cite{Ambrose, AmbroseKakutani} W. Ambrose and S. Kakutani proved a remarkable representation theorem for measurable flows which can be summarized as follows: every measurable  measure preserving flow on a Lebesgue space is isomorphic to a flow built under a function (see \cite{AmbroseKakutani} for definition). This result was latter strengthened and extended to larger classes of measurable flows (e.g. non-singular flows) by D. Rudolph \cite{Rudolph}, S. Dani \cite{Dani}, U. Krengel \cite{Krengel, Krengel2} and I. Kubo \cite{Kubo}. In \cite{Wagh} V. Wagh gave a descriptive version of Ambrose-Kakutani's theorem and more recently D. McClendon \cite{McClendon} proved a version of Ambrose-Kakutani's theorem for Borel countable-to-one semi-flows.
 
 \begin{corollary}\label{theorem:A}
Let $\phi$ be a measurable flow defined on a Lebesgue space $X$ and $\mathcal F(x)$ be the $\phi-$orbit of the point $x$. Then, given any $\phi_1-$ergodic invariant measure either
\begin{enumerate}
\item  there is a set $A \subset X$ of full $\mu$-measure such that $\mathcal F(x) \cap A$ is a discrete subset of $\mathcal F(x)$. 
Moreover, there is a natural number $k\geq 1$ such that $\mathcal F(x) \cap A$ is the $\phi_1-$orbit of exactly $k$ points; or
\item for $\mu$-almost every $x\in X$ there is a measure $\mu_{\mathcal F(x)}$ on $\mathcal F(x)$ such that
\[ \mu_{\mathcal F(x)}(\phi([0,t] \times \{x\})) = 2^{-1}t, \]
as long as $\tau \mapsto \phi(\tau,x)$ is injective on $[0,t]$, $t\geq 0$, and where $\mu_{\mathcal F(x)}$ is a $\phi_1-$invariant measure which normalized and restricted on a foliated chart of the orbit is a disintegration of $\mu$. 
\end{enumerate}
 \end{corollary}

This corollay follows as a direct consequence of the classical Ambrose-Kakutani's Theorem (see Theorem ~\ref{theorem:ambrose.kakutani} from Subsection ~\ref{subsec:Mflows}).

\section{Preliminaries on measure theory}\label{sec:preliminaries}

\subsection{Measurable partitions and Rohklin's Theorem}

Let $(X, \mu, \mathcal B)$ be a probability space, where $X$ is a compact metric space, $\mu$ a probability measure and $\mathcal B$ the Borelian $\sigma$-algebra of $X$.
Given a partition $\mathcal P$ of $X$ by measurable sets, we associate the probability space $(\mathcal P, \widetilde \mu, \widetilde{\mathcal B})$ by the following way. Let $\pi:X \rightarrow \mathcal P$ be the canonical projection, that is, $\pi$ maps a point $x$ of $X$ to the partition element of $\mathcal P$ that contains it. Then we define $\widetilde \mu := \pi_* \mu$ and 
 $\widetilde B \in \widetilde{\mathcal B}$ if and only if $\pi^{-1}(\widetilde B) \in \mathcal B$.

\begin{definition} \label{definition:conditionalmeasure}
 Given a partition $\mathcal P$. A family $\{\mu_P\}_{P \in \mathcal P} $ is a \textit{system of conditional measures} for $\mu$ (with respect to $\mathcal P$) if
\begin{itemize}
 \item[i)] given $\phi \in C^0(X)$, then $P \mapsto \int \phi \mu_P$ is measurable;
\item[ii)] $\mu_P(P)=1$ $\widetilde \mu$-a.e.;
\item[iii)] if $\phi \in C^0(X)$, then $\displaystyle{ \int_X \phi d\mu = \int_{\mathcal P}\left(\int_P \phi d\mu_P \right)d\widetilde \mu }$.
\end{itemize}
\end{definition}

When it is clear which partition we are referring to, we say that the family $\{\mu_P\}$ \textit{disintegrates} the measure $\mu$ or that it is the \textit{disintegration of $\mu$ along $\mathcal P$}.  

\begin{proposition} \label{prop:uniqueness} \cite{EW, Ro52} 
 Given a partition $\mathcal P$, if $\{\mu_P\}$ and $\{\nu_P\}$ are conditional measures that disintegrate $\mu$ on $\mathcal P$, then $\mu_P = \nu_P$ $\widetilde \mu$-a.e.
\end{proposition}


\begin{definition} \label{def:mensurable.partition}
We say that a partition $\mathcal P$ is measurable (or countably generated) with respect to $\mu$ if there exist a measurable family $\{A_i\}_{i \in \mathbb N}$ and a measurable set $F$ of full measure such that 
if $B \in \mathcal P$, then there exists a sequence $\{B_i\}$, depending on $B$, where $B_i \in \{A_i, A_i^c \}$ such that $B \cap F = \bigcap_i B_i \cap F$.
\end{definition}

\begin{theorem}[Rokhlin's disintegration \cite{Ro52}] \label{theo:rokhlin} 
 Let $\mathcal P$ be a measurable partition of a compact metric space $X$ and $\mu$ a Borel probability. Then there exists a disintegration of $\mu$ along $\mathcal P$.
\end{theorem}

\subsection{Souslin Theory} 

We list some basic properties of Souslin sets. All the results cited here can be found in  \cite[Chapter $6$]{BogachevII}.

\begin{definition}
Given a Hausdorff space $X$, a subset $A \subset X$ is called Souslin if it is the image of a complete separable metric space under a continuous mapping. We say that the Hausdorff space $X$ is a Souslin space if it is a Souslin set. By convention, we define the empty set to be a Souslin set.
\end{definition}

Observe that by definition, if $X$ and $Y$ are Hausdorff spaces and $A \subset X$ is Souslin, then given any continuous function $f:X \rightarrow Y$, the image $f(A) \subset Y$ is Souslin. Given Souslin spaces $X$ and $Y$, the product $X\times Y$ is a Souslin space and the images of a Souslin set $A\subset X\times Y$ by the projections $\pi_1:X\times Y \rightarrow X , \pi_2:X\times Y \rightarrow Y$ are Souslin sets. Notice that the image of a Borel sets even by a well behaved continuous function such as the projection may not be a Borel set. In fact this was result of a classical mistake committed by Lebesgue \cite{Lebesgue} and corrected by Souslin \cite{Souslin}.


Every Borel subset of a Souslin space is itself a Souslin space. Also, Souslin sets of Souslin spaces are preserved under Borel maps, that is, given $X$ and $Y$ Souslin spaces, $f:X\rightarrow Y$ a Borel map and $A\subset X$, $B\subset Y$ Souslin sets, then $f(A)$ and $f^{-1}(B)$ are Souslin sets. Although the complement of a Souslin set may not be a Souslin set, if the base space is Hausdorff and the complement of a Souslin set is a Souslin set it turns out that the original set is in fact a Borel set. Another interesting property of Souslin sets is that they are universally measurable sets.


\subsection{Measurable flows} \label{subsec:Mflows}
Let $(X,\mathcal B, \mu)$ be a Lebesgue space. 
In this section we briefly recall some basic notions on the structure of measurable flows and we refer the reader to \cite{Ambrose,AmbroseKakutani} for more on the subject.

\begin{definition}
A flow $\phi$ in $X$ is a one-parameter group $\{\phi_t\}$ , $-\infty < t <+\infty$, of measure preserving transformations $\phi_t :X \rightarrow X$. If $x\in X$ and $\phi$ is a flow on $X$ we say that the set $\{\phi_t(x): t\in \mathbb R\}$ is the trajectory of $x$ or the orbit of $x$ by the flow.
A flow $\phi : \mathbb R \times X \rightarrow X$ is said to be measurable if $\phi$ is a measurable function, that is, for any measurable set $Y \subset X$ the set $\{(x,t) : \phi_t(x) \in Y\}$ is a measurable set in the product space $\mathbb R \times X$ where the measure is the product of the Lebesgue measure on $\mathbb R$ and $\mu$.  
\end{definition}

Two flows $\phi = (\phi_t)_{t\in \mathbb R}$ on the Lebesgue space $(X,\mathcal B, \mu)$  and $\psi = (\psi_t)_{t\in \mathbb R}$ on the Lebesgue space $(Y,\mathcal C,\nu)$ are said to be \textit{isomorphic} if there exist invariant full measure sets $X_0 \subset X$, $Y_0 \subset Y$ and an invertible measure preserving transformation $\rho:X_0 \rightarrow Y_0$ such that
\[ \rho \circ \phi_t = \psi_t \circ \rho\]
for all $t\in \mathbb R$.

The following classical result of Ambrose-Kakutani shows that measure preserving flows on Lebesgue spaces can be represented as continuous flows on metric spaces.

\begin{theorem}[Ambrose-Kakutani \cite{AmbroseKakutani}] \label{theorem:ambrose.kakutani}
Let $\{\phi_t\}$ be a measure preserving measurable flow defined on a Lebesgue space $(X,\mathcal B, \mu)$. Then $\{\phi_t\}$ is isomorphic to a continuous flow on a separable metric space $M$ endowed with a measure $\lambda$ such that 
\begin{itemize}
\item[1)] every open set has positive $\lambda$-measure;
\item[2)] $\lambda$ is a regular measure.
\end{itemize}
\end{theorem}

%

\subsection{Measurable choice}
We finish this preliminary section with a result by R. J. Aumann \cite{Aumann}, which although comes from the Decision Theory in Economics, lies in the realm of measure theory. This result will be used in the study of some atomic case.

\begin{theorem}[Measurable Choice Theorem \cite{Aumann}] \label{theo:MCT}
Let $(T,\mu)$ be a $\sigma$-finite measure space, let $S$ be a Lebesgue space, and let $G$ be a measurable subset of $T\times S$ whose projection on $T$ is all of $T$. Then there is a measurable function $g:T\rightarrow S$, such that $(t,g(t)) \in G$ for almost all $t \in T$.
\end{theorem}

\section{Fibered spaces and disintegration} \label{sec:FiberedSpaces}
Given a continuous foliation $\mathcal F$ of a non-atomic Lebesgue probability space $X$, it is useful to look at $\mathcal F$ as fibers over a certain base space. It is not true that we can always choose a measurable set intersecting each plaque $\mathcal F(x)$ in exactly one point (the simplest example being the irrational linear foliation on the $2$-torus), so the quotient space $X/ \mathcal F$ is not always a good candidate for a base of a fibered space. In the light of this observation, instead of taking the quotient by the plaques we construct a fibered-type space over $X$ by literally attaching over each $x\in X$ the plaque $\mathcal F(x)$.

\begin{definition}\label{defi:fibered.type.space}
Given a space $X$ and a family $\mathcal P$ of subsets of $X$. We can construct a natural fibered-type space over $X$ where the fibers are given by the elements of the family $\mathcal P$. More precisely, we define the space
\[X^{\mathcal P} = \bigcup_{x\in X}\{x\}\times \mathcal P(x) \subset X\times X,\]
endowed with the $\sigma$-algebra induced by the product $\sigma$-algebra on $X\times X$.

We call $X^{\mathcal P}$ the $(X,\mathcal P)$-fibered space or simply the $\mathcal P$-fibered space. Each subset $\{x\}\times \mathcal P(x) \subset X^{\mathcal P}$ is called the fiber of $x$ on $X^{\mathcal P}$. 
\end{definition}

Given a continuous foliation $\mathcal F$ of a non-atomic Lebesgue probability space $(X,\mu)$, consider a local chart 
\[\varphi_x:(0,1)\times (0,1)^k \rightarrow U\]
of $\mathcal F$. The partition $\mathcal V = \{ \{x\} \times (0,1)^k \}$ is a measurable partition of $(0,1)\times (0,1)^k$ with respect to any Borel measure $\mu$ due to the separability of $(0,1)^k$. Hence on a local chart the partition given by the segments of leaves $\varphi_x(\{x\} \times (0,1)^k )$ forms a measurable partition on $U$ for any Borel measure on $U$. That means we can always disintegrate a measure on a local chart. Although the partition by the leaves of a foliation is not necessarily a measurable partition the next result allow us to say that the disintegration of a measure is atomic on the leaves, or it is absolutely continuous to Lebesgue on the leaves, since these properties persists independent of the foliated box one uses to disintegrate.


\begin{proposition} \label{prop:disintegration.unbounded}
If $U_1$ and $U_2$ are described by the local charts $\varphi_{x_1}$ and $\varphi_{x_2}$ of $\mathcal F$ respectively, then the conditional measures $\mu_x^{U_1}$ and $\mu_x^{U_2}$, of $\mu$ on $U_1$ and $U_2$ respectively, coincide up to a constant on $U_1 \cap U_2$.
\end{proposition}
\begin{proof}
It follows from \cite[Proposition 5.17]{El.pisa}.
\end{proof}

\begin{definition}
We say that a probability $\mu$ has atomic disintegration with respect to a foliation if its conditional measures on any foliated box are sum of Dirac measures.
\end{definition}

\begin{remark}\label{remark:class.measures}
Consider the classical volume preserving Kronecker irrational flow on the the torus $\mathbb T^2$. Let $\mathcal F$ be the continuous foliation given by the orbits of this flow, it follows that this is not a measurable partition in the sense of Definition \ref{def:mensurable.partition}. Hence, we cannot apply the Rohklin's disintegration Theorem \ref{theo:rokhlin} even on the apparently well-behaved continuous foliations. But we may always disintegrate locally and compare two local disintegrations by the above result. The proposition above implies that we can talk about disintegration of a measure over a foliation even if it does not form a measurable partition, as long as we have in mind that for a disintegration we understand that on a plaque there is a class of conditional measures which differ up to a multiplication of a constant, we denote this system of conditional measures as $\{[\mu_x]\}$. More precisely for a given foliation $\mathcal F$ on each plaque $\mathcal F(x)$ there is a family of measures $[\mu_x]$ defined on $\mathcal F(x)$ such that if $\eta \in [\mu_x]$ then $\mu_x = \alpha \eta$ for some positive constant $\alpha \in \mathbb R$. And on a foliated box if one normalizes these measures they form a disintegration of the measure $\mu$ in this foliated box. 
\end{remark}

\subsection{Disintegration on the unitary fibered space}
In this section the foliation $\mathcal F$ comes from the orbits of a continuous flow $\phi$ on a separable metric space $X$. Hence $\mathcal F(x)$ is the orbit of $x$ through the flow $\phi$. Denote $B_{\mathcal F}(x,r):=\phi((-r,r)\times \{x\})$ and consider the family of sets \[\mathcal F^1 = \{ \mathcal F^1(x):= \{x\} \times B_{\mathcal F}(x,1)\}_x.\] 
For convenience, denote by $X_1^{\mathcal F}$ the $(X,\mathcal F^1)-$fibered space, that is,
 \[X_1^{\mathcal F} = \bigcup_{x \in X}\mathcal F^1(x).\]

\begin{lemma} \label{lemma:measurable.for.unbounded}
The partition of $X_1^{\mathcal F}$ by the verticals $\mathcal F^1(x)$ is a measurable partition with respect to any measure on $X_1^\mathcal{F}$.
\end{lemma}
\begin{proof}
Let $\{U_i\} \subset X$ be a countable basis of open sets of $X$. By the definition of $\mathcal F$, the $\mathcal F$ saturation of $U_i$ is given by $ \phi((-\infty,+\infty) \times U_i)$, which is a measurable set since the flow is continuous. Let $V_i:= (U_i \times \mathcal F(U_i)) \cap X_1^{\mathcal F}$. Each $V_i$ is a measurable set in $X^{\mathcal F}_1$. Now, it is easy to see that each fiber can be written as intersection of sets of the countable family of sets $\{V_i\}$ or its complement.

\end{proof}

\begin{proposition} \label{prop:unbounded.measurable}
For each $x\in X$ denote by $\mu^1_x$ the measure on the equivalence class  $[\mu_x]$ (as defined on Remark \ref{remark:class.measures}) such that $\mu^1_x$ is a probability measure when restricted to $B_{\mathcal F}(x,1)$. Then 
\[x\mapsto \mu^1_x\]
is a measurable map, that is, given any measurable set $W\subset X$ the function
\[x\mapsto \mu^1_x(W)\]
is a measurable function.
\end{proposition}
\begin{proof}
On the fibered space $X^{\mathcal F}_1$ consider the measure $\widetilde{\mu}$ defined by
$$\widetilde \mu (\widetilde{A}) = \int_X \mu^1_x(\widetilde{A}_x) d\mu(x), $$
 for any measurable set $\widetilde{A} \subset X^{\mathcal F}_1$, where $\widetilde{A}_x = \{y\in B_{\mathcal F}(x,1): (x,y) \in \widetilde{A}\}$.
Since the vertical partition on $X^{\mathcal F}_1$ is a measurable partition by Lemma \ref{lemma:measurable.for.unbounded} the probability measure $\widetilde \mu$ has a Rohklin disintegration along the leaves for which the conditional measures varies measurably on the base point. By uniqueness and by the definition of $\widetilde \mu$ we have that the conditional measure on the plaque $\{x\} \times B_{\mathcal F}(x,1)$ is exactly $\mu^1_x$. By the properties of the Rohklin disintegration it follows that given any measurable set $\widetilde{W} \subset X^{\mathcal F}_1$ we have that $x\mapsto \mu^1_x(\widetilde{W}_x)$ is a measurable function.
Given any measurable set $W\subset X$ let 
\[\widetilde{W}:= \bigcup_{x\in W} \{x\}\times [W\cap B_{\mathcal F}(x,1)].\]
Thus $\widetilde{W}_x = W\cap B_{\mathcal F}(x,1)$ and then we have that
\[x\mapsto \mu^1_x(W \cap B_{\mathcal F}(x,1)) = \mu^1_x(W)\]
is a measurable function on $x$ as we wanted to show.
\end{proof}


\begin{proposition} \label{prop:disinmeasurable2}
For each $r\in (0,\infty)$ the function
\[x\mapsto \mu^1_x(B_{\mathcal F}(x,r))\]
is a measurable function.
\end{proposition}
\begin{proof}
For each fixed $0\leq r \leq 1$, define the following $r$-top function
\[f_r: X \rightarrow X_1^{\mathcal F},  \quad f_r(z) = (z,\phi_r(z)).\]
Let
\[W := \bigcup_{z\in X} [f_{-r}(z),f_r(z)]_z,\]
where $[x, y]_z$ denotes the closed vertical segment connecting $x$ and $y$ on $\mathcal F(z)$. Observe that $W$ is Borel since it is a compact set. By the measurability of $x\mapsto \mu^1_x(W)$ we have that $z\mapsto \mu^1_z(B_{\mathcal F}(z,r))$ is a measurable function. Consequently
\[z\mapsto \mu^1_z(B_{\mathcal F}(z,r))\]
is a measurable function.
Since $x\mapsto \mu_{x}^1(B_{\mathcal F}(x,r_0))$, for every $0\leq r_0 \leq 1$ fixed, is $\phi_1-$invariant we have conclude that for any $r\in (0,\infty)$ the function $x\mapsto \mu_{x}^1(B_{\mathcal F}(x,r))$ is measurable.
\end{proof}

\begin{corollary} \label{cor:jointlymeasurable2}
If $\{[\mu_x]\}$ is a non-atomic system of conditional measures then, for each typical $x\in X$ the function
\[r\mapsto \mu^1_x(B_{\mathcal F}(x,r))\]
is continuous. Furthermore the function
\[(x,r) \mapsto \mu^1_x(B_{\mathcal F}(x,r))\]
is jointly measurable.
\end{corollary}
\begin{proof}
Let $x\in X$ be a $\mu$-typical point, hence $\mu^1_x$ is a non-atomic measure on $\mathcal F(x)$. First, let us prove that $r \mapsto \mu^1_x(B_{\mathcal F}(x,r))$ is a continuous function. Let $y_n \in \mathcal F(x)$ and $\varepsilon_n \searrow \varepsilon \in (0,\infty)$, hence  $\mu^1_{x}(B_{\mathcal F}(x,\varepsilon_n)) = \mu^1_x(B_{\mathcal F}(x,\varepsilon)) + \mu^1_x(B_{\mathcal F}(x,\varepsilon_n)\setminus B_{\mathcal F}(x,\varepsilon))$. Because $\mu^1_x$ is nonatomic 
$$\lim_{n\rightarrow \infty}\mu^1_x(B_{\mathcal F}(x,\varepsilon_n)\setminus B_{\mathcal F}(x,\varepsilon)) =0.$$ 
Then,
$$\mu^1_{x}(B_{\mathcal F}(x,\varepsilon_n))  \rightarrow  \mu^1_x(B_{\mathcal F}(x,\varepsilon)).$$
By Proposition \ref{prop:disinmeasurable2} we know that $x \mapsto \mu^1_x(B_{\mathcal F}(x,r))$ is a measurable function, therefore the function $(x,r) \mapsto \mu^1_{x}(B_{\mathcal F}(x,r))$ is a Carath\'eodory function (i.e. measurable in one variable and continuous in the other, see \cite[Definition 4.50]{InfDimAna}), in particular it is a jointly measurable function \cite[Lemma 4.51]{InfDimAna}.
\end{proof}

\subsection{The leafwise measure distortion}

The last concept we will introduce in this section is the concept of leafwise measure distortion.

\begin{definition} \label{defi:distortion2}
Let $(X,\mu)$ be a non-atomic Lebesgue space and $\mathcal F$ be a continuous foliation of $X$ induced by the orbits of a continuous flow $\phi_t$. Let $\{[\mu_{x}]\}$ denote the system of equivalence classes of conditional measures along $\mathcal F$. We define the upper and lower $\mu$-distortion at $x$ respectively by
\[\overline{\Delta(\mu)}(x):= \limsup_{\varepsilon \rightarrow 0}\frac{\mu^1_x(B_{\mathcal F}(x,\varepsilon))}{\varepsilon} ,\quad \underline{\Delta(\mu)}(x):= \liminf_{\varepsilon \rightarrow 0}\frac{\mu^1_x(B_{\mathcal F}(x,\varepsilon))}{\varepsilon},\]
where $\mu^1_x$ is taken to be the measure on the class of $[\mu_{x}]$ which gives weight one to $B_{\mathcal F}(x,1)$.
If the upper and lower distortions at $x$ are equal then we just call it the $\mu$-distortion at $x$ and denote by
\[\Delta(\mu)(x):= \lim_{\varepsilon \rightarrow 0}\frac{\mu^1_x(B_{\mathcal F}(x,\varepsilon))}{\varepsilon}.\]
\end{definition}

\section{Proof of the main result} \label{sec:DDL}

We proceed to the proof of our main result, Theorem \ref{theo:continuous.flow}, but first we provide a sketch of its proof.

\subsection{Sketch of the proof of Theorem \ref{theo:continuous.flow}}
The proof will be made in two steps. The first, and easy case, is the atomic case. The second case, the non-atomic case is the one where the main ideas appear. 

The first observation is that ergodicity implies that the upper (resp. lower) $\mu-$distortion at $x$ is constant almost everywhere. Then, using the $\varphi_1-$invariance of the family $\{B_{\mathcal F}(x,r)\}$ and the ergodicity of the measure, we obtain some uniformity on the upper (resp. lower) $\mu-$distortion in the sense that along a certain sequence $(\varepsilon_k)_k$, $\varepsilon_k \rightarrow 0$, the ratios appearing in Definition \ref{defi:distortion2} converge to the upper (resp. lower) $\mu-$distortion with the same rate for almost every point $x \in X$ . This is proven in Lemmas \ref{lema:sequenciaboa} and  \ref{lema:sequenciaboa2}. Once proven this uniformity of the upper (resp. lower) distortion, we turn our attention to the set of all points $\overline{\Pi}$ (resp. $\underline{\Pi}$) where such uniformity occurs and its topological characteristics when restricted to a plaque. To be more precise, we prove in Lemma \ref{lemma:closed} that the set of points for which the uniforme distortions occurs is closed in each plaque intersecting it.
The last step consists of analyzing the set $D$ of points $x$ for which $\overline{\Pi}$ is dense in $\mathcal F(x)$, that is, $\overline{\Pi}\cap \mathcal F(x)=\mathcal F(x)$. $D$ is $\varphi_1$-invariant thus it has full or zero measure. If it has full measure then the denseness of $\overline{\Pi}$ on the plaques $\mathcal F(x), x\in D$, allows us to extend the uniform upper distortion to every point on the respective plaque (i.e. orbit). Using the uniformity at every point we prove that the upper distortion is a constant times the $\mu_x$ measure of the set $B_{\mathcal F}(x,1)$ on the plaque $\mathcal F(x)$. Applying the same argument for the set $\underline{\Pi}$ where the lower distortion is uniform we get to the same equality and conclude that the upper and lower distortion are equal, thus the limit converges and we actually have a well defined distortion. Using this fact we prove in Lemma \ref{lemma:equallebesgue} that $\mu^1_x$ is a constant times the natural measure induced by the flow on the orbits.
If $D$ has zero measure then almost every plaque has pieces of open intervals in it which are in the complement of the set $\overline{\Pi}$. We use this holes to show that atoms should appear, which yields an absurd.

\subsection{Proof of Theorem \ref{theo:continuous.flow}}

To simplify notation we denote $f:=\phi_1$.

First let us deal with the case where $\mu$ itself has atoms, that is, there is a countable subset $Z\subset X$ such that $\mu(\{z\})>0$ for any $z\in Z$. Since $f$ is ergodic and $Z$ is $f$-invariant we have $\mu(Z)=1$. Hence the second item of the theorem is satisfied. We may now assume the measure $\mu$ itself is atomless.

Let $Per(\phi)$ to be the set of periodic orbits of the flow $\phi$. First let us assume that $\mu(Per(\phi))=0$ and break the proof in two cases (the \textit{atomic case} and the \textit{non-atomic case}). We deal with $\mu(Per(\phi))>0$ by the end of the proof. Also recall that $\mathcal F$ is the foliation whose plaques are the orbits of the flow and  $B_{\mathcal F}(x,r):=\phi((-r,r)\times \{x\})$. 


\vspace{0.3cm}
\textbf{The atomic case:} Assume that $\mu$ has atomic disintegration over $\mathcal F$.

Consider the measurable function $g_r: x \mapsto \mu_x^1(B_{\mathcal F}(x,r))$. Now define the weight map
$$w: x \mapsto \mu_x^1(\{x\}).$$
This is a measurable map because $w(x) = \lim_{r \rightarrow 0} g_r(x)$ and pointwise limit of measurable functions is a measurable function.

Now consider the invariant set $w^{-1} ((0,\delta) )$ of atoms whose weight is less then $\delta$. Ergodicity implies that this set has zero or one measure. Thus, there exists a real number $\delta_0>0$ such that each atom has weight $\delta_0$ and, consequently, each plaque has the same number of atoms $k_0 = 1/\delta_0$. 

Hence we have proved statement $(2)$ of Theorem \ref{theo:continuous.flow}.

\vspace{0.3cm}
\textbf{Non-atomic case}: We now assume that the disintegration is not atomic.

Let $\{[\mu_x]\}$, as in Remark \ref{remark:class.measures}, be the equivalence classes of the conditional measures coming from the Rokhlin disintegration of $\mu$ along the leaves of $\mathcal F$.
%
Observe that $\mu_x^1(B_{\mathcal F}(x,\varepsilon)) >0 $ for every $x \in \operatorname{Supp}_{\mathcal F}(\mu^1_x)$ (where the support here is inside $\mathcal F(x)$). Thus, it makes sense to evaluate the upper and lower unitary distortions. Also observe that, a priori, $\overline{\Delta}(x)$ and $\underline{\Delta}(x)$ could be infinity. In any case these functions are well-known to be measurable functions.
Also note that both $\overline{\Delta}(x)$ and $\underline{\Delta}(x)$ are $f$-invariant maps because
\[f_{*}\mu^1_x = \mu^1_{f(x)} \;\text{ and }\;f(B_{\mathcal F}(x,\varepsilon)) = B_{\mathcal F}(f(x),\varepsilon).\]
By ergodicity of $f$ it follows that both are constant almost everywhere, let us call these constants by $\overline{\Delta}$ and $\underline{\Delta}$. That is, for almost every $x$:
\begin{equation}\label{eq:delta}
\overline{\Delta}(x)  = \overline{\Delta} ,\quad \text{ and } \underline{\Delta}(x)  = \underline{\Delta}.
\end{equation}
Let $D$ be a (full measure) set of points $x$ for which \eqref{eq:delta} occurs.

\begin{lemma} \label{lema:sequenciaboa} 
If $\overline{\Delta}$ is finite, there exists a sequence $\varepsilon_k\rightarrow 0$, as $k\rightarrow +\infty$, and a full measure subset $R \subset D$ such that
\begin{itemize}
\item[i)] $R$ is $f$-invariant;
\item[ii)]for every $x \in R$, then
\begin{equation}\label{eq:uniform}
\left| \frac{ \mu^1_x(B_{\mathcal F}(x,\varepsilon_k))}{\varepsilon_k} - \overline{\Delta}  \right|   \leq \frac{1}{k};\end{equation}
\end{itemize}
An analogous result holds if instead of $\overline{\Delta}$ we consider $\underline{\Delta}$.
\end{lemma}
\begin{proof}
Since $\overline{\Delta}(x) = \overline{\Delta}$ for every $x \in D$ and $k\in \mathbb N^{*}$ define
\[\varepsilon_k(x):= \sup \left\{\varepsilon: \left| \frac{ \mu^1_x(B_{\mathcal F}(x,\varepsilon))}{ \varepsilon} - \overline{\Delta}  \right| +\varepsilon  \leq \frac{1}{k} \right\}.\]
Observe that $\varepsilon_k(x)$ exists because since the $\limsup$ is $\overline{\Delta}$ we can take a sequence $\varepsilon_l(x) \rightarrow 0$ such that the ratio given approaches $\overline{\Delta}$.

\noindent {\bf Claim:}
The function $\varepsilon_k(x)$ is a measurable for all $k \in \mathbb N$.
\begin{proof}
Observe that since $\mu_x$ is not atomic we have
\[\varepsilon_k(x) = \lim_{n\rightarrow \infty} \varepsilon^n_k(x)\]
where
\[\varepsilon^n_k(x) = \sup \left\{\varepsilon: \left| \frac{ \mu_x^1(B_{\mathcal F}(x,\varepsilon))}{ \varepsilon} - \overline{\Delta}  \right|  +\varepsilon < \frac{1}{k} +\frac{1}{n} \right\}.\]
So, it is enough to prove that $\varepsilon_k^n(x)$ is measurable on $x$. 

Define \[g(x,\varepsilon) = \left| \frac{ \mu_x^1(B_{\mathcal F}(x,\varepsilon))}{ \varepsilon} - \overline{\Delta}  \right| + \varepsilon.\]
By Corollary \ref{cor:jointlymeasurable2}, for any typical $x\in M$ the function $g(x,\cdot):(0,\infty) \rightarrow \mathbb (0,\infty)$ is continuous.
Let $\varepsilon>0$ be fixed and let us prove that $g(\cdot, \varepsilon):M \rightarrow \mathbb (0,\infty)$ is a measurable function. By Proposition \ref{prop:disinmeasurable2} we know that $x\mapsto  \mu_x^1(B_{\mathcal F}(x,\varepsilon))$ is a measurable function, therefore $g(\cdot, \varepsilon)$ is measurable function.

%
%



Given any $k\in \mathbb N$, $k>0$, the continuity of $g(x,\cdot)$ implies that 
\[\varepsilon_k^{-1}((0,\beta))=\{x: \varepsilon_{k}(x) \in (0,\beta)\}  = \bigcap_{r\geq b, r\in \mathbb Q} g(\cdot, r)^{-1}([1/k, +\infty)). \]
Therefore $\varepsilon_k^{-1}((0,\beta))$ is measurable and consequently $\varepsilon_k$ is a measurable function for every $k$.
\end{proof}

Note that $\varepsilon_k(x)$ is $f$-invariant. Thus, by ergodicity, let $R_k$ be a full measure set such that $\varepsilon_k(x)$ is constant equal to $\varepsilon_k$.
It is easy to see that the sequence $\varepsilon_k$ goes to $0$ as $k$ goes to infinity. Take $\widetilde{R}:= \bigcap_{k=1}^{+\infty} R_k$.
Since each $R_k$ has full measure, $\widetilde{R}$ has full measure and clearly satisfies what we want for the sequence $\{\varepsilon_k\}_{k}$.
Finally, take $R = \bigcap_{-\infty}^{+\infty} f^i(\widetilde{R})$. $R$ is $f$-invariant, has full measure and satisfies $(i)$ and $(ii)$. 

\end{proof}

Now consider the following set
\[\overline{\Pi} := \bigcup_{x \in R} \overline{\Pi}_x.\] 
where
\[\overline{\Pi}_x:= \left\{y \in \mathcal F(x):\left| \frac{ \mu^1_x(B_{\mathcal F}(y,\varepsilon_k))}{\varepsilon_k} - \overline{\Delta}  \right|   \leq \frac{1}{k}, \forall k\geq 1 \right\},\]
similarly we define $\underline{\Pi}_x$ and $\underline{\Pi}$ with $\underline \Delta$ in the role of $\overline \Delta$.

\begin{lemma} \label{lemma:closed}
For every $x \in R$ the set $\overline{\Pi}_x$ is closed in the plaque $\mathcal F(x)$.
\end{lemma}
\begin{proof}
Let $y_n \rightarrow y$, $y_n \in \overline{\Pi}_x$, $y\in \mathcal F(x)$. To prove that $y\in \overline{\Pi}_x$ it is enough to show that 
 \[\lim_{n\rightarrow \infty} \mu_x^1(B_{\mathcal F}(y_n,\varepsilon_k)) = \mu_x^1(B_{\mathcal F}(y,\varepsilon_k)). \]
Given any $k\in \mathbb N$, since $\mu_x$ is not atomic we have that 
\begin{eqnarray*}
 \mu^1_x(\partial B_{\mathcal F}(y,\varepsilon_k)) = \mu^1_x(\phi(-\varepsilon_k,y) \cup \phi(\varepsilon_k,y)) =0
\end{eqnarray*}
and
\begin{eqnarray*}
 \mu^1_x(\partial B_{\mathcal F}(y_n,\varepsilon_k)) = \mu^1_x(\phi(-\varepsilon_k,y_n) \cup \phi(\varepsilon_k,y_n)) =0, \forall n \in \mathbb N,
\end{eqnarray*}
where $\partial B_{\mathcal F}$ denotes the boundary of the set inside the leaf.

Now, let $B_n:=B_{\mathcal F}(y_n,\varepsilon_k) \Delta B_{\mathcal F}(y,\varepsilon_k)$ where $Y\Delta Z$ denotes the symmetric diference of the sets $Y$ and $Z$. Observe that, by passing to a subsequence of $y_n$ if necessary, we have $B_n \supset B_{n+1}$, for every $n \geq 1$. Thus
\begin{align*}
  \lim_{n\rightarrow \infty}  \mu^1_x(B_n) & =  \lim_{n\rightarrow \infty}  \mu^1_x \left(\bigcap_{n} B_n \right) \\
 &=   \lim_{n\rightarrow \infty}  \mu^1_x(\{\phi(-\varepsilon_k, y) , \phi(\varepsilon_k,y) \}) \\
 &=  0.\end{align*}
 Therefore $\lim_{n\rightarrow \infty} \mu_x^1(B_{\mathcal F}(y,\varepsilon_k) \setminus B_{\mathcal F}(y_n,\varepsilon_k)) = \lim_{n\rightarrow \infty} \mu_x^1(B_{\mathcal F}(y_n,\varepsilon_k) \setminus B_{\mathcal F}(y,\varepsilon_k))= 0$ and consequently 
 \[\lim_{n\rightarrow \infty} \mu_x^1(B_{\mathcal F}(y_n,\varepsilon_k)) = \mu_x^1(B_{\mathcal F}(y,\varepsilon_k)), \]
 as we wanted to show.
 
%
%
\end{proof}
An analogous result is true for $\overline{\Delta}$.
\begin{lemma} \label{lema:sequenciaboa2}
If $\overline{\Delta}$ is infinity, there exists a sequence $\varepsilon_k\rightarrow 0$, as $k\rightarrow +\infty$, and a full measure subset $R^{\infty} \subset D$ such that
\begin{itemize}
\item[i)] $R^{\infty}$ is $f$-invariant;
\item[ii)]for every $x \in R^{\infty}$ we have 
\begin{equation}\label{eq:infity}
 \frac{\mu^1_x(B_{\mathcal F}(x,\varepsilon_k))}{\varepsilon_k} \geq k .\end{equation}
\end{itemize} 
An analogous result holds if instead of $\overline{\Delta}$ we consider $\underline{\Delta}$.
\end{lemma}

Analogously to what we have done for the finite case, define
\[\overline{\Pi}^{\infty}_x:= \left\{y \in \mathcal F(x): \frac{ \mu^1_x(B_{\mathcal F}(x,\varepsilon_k))}{\varepsilon_k} \geq k , \forall k\geq 1\right\},\]
and
\[\overline{\Pi}^{\infty} := \bigcup \overline{\Pi}^{\infty}_x.\]
Similarly we define $\underline{\Pi}^{\infty}_x$ and $\underline{\Pi}^{\infty}$.

\begin{lemma}
If $\overline{\Delta}$ (resp. $\underline{\Delta}$) is infinity then for every $x \in R$ the set $\overline{\Pi}^{\infty}_x$ (resp. $\underline{\Pi}^{\infty}_x$) is closed on the plaque $\mathcal F(x)$.
\end{lemma}
\begin{proof}
Analogous to the proof of Lemma \ref{lemma:closed}.
\end{proof}

\begin{lemma}\label{lemma:two.sets}
 If $\overline{\Delta}$ is finite, then there are Borel sets $Q$ and $G$ such that
 \begin{itemize}
  \item[i)] $f(Q)=Q$ and $f(G)=G$;
  \item[ii)] $Q \cap G = \emptyset$;
  \item[iii)] $\mu(Q\cup G)=1$;
  \item[iv)] if $x \in Q$, then for $\varepsilon_k$ as in Lemma \ref{lema:sequenciaboa} then
  \[\left| \frac{ \mu_x^1(B_{\mathcal F}(x,\varepsilon_k))}{\varepsilon_k} - \overline{\Delta}  \right|   \leq \frac{1}{k};\]  
  \item[v)] if $x \in G$, then there exists $k_0 \in \mathbb N$ such that 
    \[\left| \frac{ \mu^1_x(B_{\mathcal F}(x,\varepsilon_{k_0}))}{\varepsilon_{k_0}} - \overline{\Delta}  \right|   > \frac{1}{k_0}.\]  
 \end{itemize}
\end{lemma}
\begin{proof}
Consider $\overline{\Pi}$ as defined above. Take any $x\in \overline{\Pi}^c$, that is, there exists $k\geq 1$ such that
 \[\left| \frac{ \mu^1_x(B_{\mathcal F}(x,\varepsilon_{k}))}{\varepsilon_{k}} - \overline{\Delta}  \right|   > \frac{1}{k}.\]  
By the measurability of $x\mapsto \mu^1_x(B_{\mathcal F}(x,\varepsilon_k))$ proved in Proposition \ref{prop:disinmeasurable2} and Lusin's Theorem we can take a compact set $G_1$ where this function varies continuously. Thus, there exists an open set $G_2$ such that for every $y\in G_2\cap G_1$ we have
 \[\left| \frac{ \mu^1_x(B_{\mathcal F}(x,\varepsilon_{k}))}{\varepsilon_{k}} - \overline{\Delta}  \right|   > \frac{1}{k}.\]  
Define $G = \bigcup_{n\in \mathbb Z}f^n(G_2\cap G_1)$.

Let $x \in \overline{\Pi}$. For each $n\in \mathbb N$ we have
 \[\left| \frac{ \mu^1_x(B_{\mathcal F}(x,\varepsilon_k))}{\varepsilon_k} - \overline{\Delta}  \right| < \frac{1}{k}+\frac{1}{n}.\]
 Using again Proposition \ref{prop:disinmeasurable2}, Lusin's Theorem and the invariance of $\mu_x$ by $f$, we find a sequence of nested Borel sets $ \ldots Q_{n+1} \subset Q_n \subset Q_{n-1} \subset ...\subset Q_1$ such that $f(Q_n)=Q_n$, $n\geq 1$ and for all $y\in Q_n$ we have
 \[\left| \frac{ \mu^1_y(B_{\mathcal F}(y,\varepsilon_k))}{\varepsilon_k} - \overline{\Delta}  \right| < \frac{1}{k}+\frac{1}{n}.\]
By Lemma \ref{lema:sequenciaboa} we have $\mu(Q_n)=1$ for every $n$.
Take $Q:=\bigcap_{n=1}^{\infty}Q_n$. Then $Q$ is an $f$-invariant Borel set and $\mu(Q)=1$.
Therefore $\mu(Q\cup G) = 1$. Also, it is clear that $Q\cap G = \emptyset$ and we conclude the proof of the lemma.
\end{proof}
Consider the following measurable set 
$$D := \mathcal F(Q) \setminus \mathcal F(G).$$
Equivalently
\[D = \{x \in \mathcal F(G \cup Q) : \overline{\Pi}_x \cap \mathcal F(x) = \mathcal F(x)\},\]
that is, $D$ is the set of all points whose plaque is fully inside $\overline{\Pi}_x$.

In the sequel of the proof we will need the following counting lemma.

\begin{lemma} \label{lemma:aux}
Let $r>0$ be a fixed real number and  $x\in D$ an arbitrary point. Let $a_i := \varphi_{2i r}(\varphi_{-1}(x))$ and $b_i := \varphi_{2ir}(x)$ for $i=1,2,..., l$ where $l=\left \lfloor \frac{1}{2}\left(\frac{1}{r}-1\right)\right \rfloor$. Then
\begin{equation}\label{eq:statement}
\sum_{i=1}^{l}\mu_{x}^{1}(B_{\mathcal F}[a_i,1]) + \sum_{i=1}^{l}\mu_{x}^{1}(B_{\mathcal F}[b_i,1]) = 2l.\end{equation}
\end{lemma}
\begin{proof}
To simplify the notation, for $s>0$ we will write $[x,\varphi_s(x)]$ to denote the set $\{\varphi_t(x): 0\leq t \leq s\}$. With this notation we can write
\begin{equation*}
[\varphi_{-1}(x),x] = [\varphi_{-1}(x), a_1] \cup [a_1,a_2] \cup \ldots \cup [a_{l-1}, a_l] \cup [a_l, \varphi_{2(l+1)r-1}(x)] \cup [\varphi_{2(l+1)r-1}(x),x] 
\end{equation*}
Denote $J_0 := [\varphi_{-1}(x), a_1] $, $J_i:=[a_i,a_{i+1}]$ for $1\leq i \leq l-1$, $J_l :=  [a_l, \varphi_{2(l+1)r-1}(x)] $ and $J_{l+1} =  [\varphi_{2(l+1)r-1}(x),x] $. Thus we can rewrite
\begin{equation}\label{eq:part1}
[\varphi_{-1}(x),x] = J_0 \cup \ldots J_{l+1}.
\end{equation}
Now, by applying $\varphi_1$ to \eqref{eq:part1} we can write
\begin{align}\label{eq:part2}
[x,\varphi_1(x)] = & [x, b_1] \cup [b_1,b_2] \cup \ldots \cup [b_{l-1}, b_l] \cup [b_l, \varphi_{2(l+1)r}(x)] \cup [\varphi_{2(l+1)r}(x),\varphi_1(x)] \\
=& \varphi_1(J_0) \cup \ldots \varphi_1(J_{l+1}).
\end{align}
Also as a consequence of \eqref{eq:part1} we can write
\begin{equation} \label{eq:part3}
[\varphi_{-2}(x), \varphi_2(x)] = \varphi_{-1}(J_0) \cup \ldots \cup \varphi_{-1}(J_{l+1}) \cup [\varphi_{-1}(x),x] \cup  [x,\varphi_1(x)] \cup \varphi_{2}(J_0) \cup \ldots \cup \varphi_{2}(J_{l+1}) . \end{equation}

Now, observe that each term involved in the sums on the left side of \eqref{eq:statement} can be written as the sum of the $\mu_x^1-$measure of sets of the forms involved on the equations \eqref{eq:part1}, \eqref{eq:part2} and \eqref{eq:part3}. Lets count how many times each of this sets appears on the left side of \eqref{eq:statement}. 
\begin{itemize}
\item Observe that the set $\varphi_{-1}(J_0)$ is not contained in any of the sets $B_{\mathcal F}[a_i,1]$, $B_{\mathcal F}[b_i,1]$, thus it does not appears on \eqref{eq:statement}. However, the set $\varphi_1(J_0)=[x,\varphi_1(a_1)]$ is contained in all of the sets $B_{\mathcal F}[a_i,1], B_{\mathcal F}[b_i,1]$, thus is appears on $2l$ times on the equation \eqref{eq:statement}. Thus, $\mu_x^1(\varphi_1(J_0))$ appears exactly $2l$ times on \eqref{eq:statement}. 
\item For any $1\leq i \leq l+1$ the set $\varphi_{-1}(J_i)$ appears on each of the terms $B_{\mathcal F}[a_j,1]$, $j=1,...,i$, that is, it appears $i$ times on \eqref{eq:statement}. On the other hand the set $\varphi_1(J_i)$ appears $2l-i$ times as it does not belong only to the sets $B_{\mathcal F}[a_j,1]$, $j=1,...,i$. By the fact that $\varphi_1$ preserves $\mu_x^1$ we know that $\mu_x^1(\varphi_{-1}(J_i)) =\mu_x^1(\varphi_1(J_i))$ and then we can say that $\mu_x^1(\varphi_1(J_i))$ appears exactly $2l$ times on \eqref{eq:statement}
\item By symmetry we can see that the terms $\mu_x^1(J_i)$ also appears exactly $2l$ times each. 
\end{itemize}
Thus we have 
\begin{align*}
\sum_{i=1}^{l}\mu_{x}^{1}(B_{\mathcal F}[a_i,1]) +  \sum_{i=1}^{l}\mu_{x}^{1}(B_{\mathcal F}[b_i,1]) = & \\
= & l \cdot \left(2l\cdot \sum_{i=0}^{l+1} \mu_x^1(J_i) + 2l\cdot \sum_{i=0}^{l+1} \mu_x^1(\varphi_{1}(J_i)) \right) \\
= & 2l \cdot \mu_{x}^1(B_{\mathcal F}[x,1]) = 2l.\end{align*}
as we wanted to show.
\end{proof}

\begin{lemma} \label{lemma:aux2}
Let $r>0$ be a fixed real number and  $x\in D$ an arbitrary point. Let $a_i := \varphi_{2i r}(\varphi_{-1}(x))$ and $b_i := \varphi_{2ir}(x)$ for $i=1,2,..., l+1$ where $l=\left \lfloor \frac{1}{2}\left(\frac{1}{r}-1\right)\right \rfloor$. Then
\begin{equation}\label{eq:statement}
\sum_{i=1}^{l+1}\mu_{x}^{1}(B_{\mathcal F}[a_i,1]) + \sum_{i=1}^{l}\mu_{x}^{1}(B_{\mathcal F}[b_i,1]) = 2l+2.\end{equation}
\end{lemma}
\begin{proof}
The proof is identical to the proof of Lemma \ref{lemma:aux}.
\end{proof}

\vspace{.3cm}
\textbf{Case 1: $D$ has full measure.} First of all, we will prove that in this case we must have $\underline{\Delta}\leq \overline{\Delta}<\infty$. Assume that $\overline{\Delta}=\infty$. Consider a typical fiber $\mathcal F(x)$ with $x\in D$ and take any $k\geq 1$ fixed. On Lemma \ref{lemma:aux} take $r:=\varepsilon_k$ and let $l, a_i, b_i$, $1\leq i\leq l$ be as in the statement of the respective lemma. 
For each $1\leq i \leq l$ we have
\begin{equation}
 \frac{\mu^1_{a_i}(B_{\mathcal F}[a_i,\varepsilon_k]))}{\varepsilon_k} \geq k  \Rightarrow  \mu^1_{a_i}(B_{\mathcal F}[a_i,\varepsilon_k]) \geq k \varepsilon_k,
\end{equation}
and similarly we obtain
\begin{equation}
\mu^1_{b_i}(B_{\mathcal F}[b_i,\varepsilon_k]) \geq k \varepsilon_k.\end{equation}
Now observe that 
\begin{align} \label{eq:xai} 
\mu_x^1(B_{\mathcal F}[a_i,\varepsilon_k])) = & \mu_x^1(B_{\mathcal F}[a_i,1])) \cdot \mu_{a_i}^1(B_{\mathcal F}[a_i,\varepsilon_k]) \\
\mu_x^1(B_{\mathcal F}[b_i,\varepsilon_k])) = & \mu_x^1(B_{\mathcal F}[b_i,1])) \cdot \mu_{b_i}^1(B_{\mathcal F}[b_i,\varepsilon_k]) \label{eq:xbi}
\end{align}
Taking the sum over $i$ we have
\begin{align*}
 \mu_x^1(B_{\mathcal F}[x,1]) \geq \sum_{i=1}^{l} \mu_x^1(B_{\mathcal F}[a_i,\varepsilon_k]) +& \sum_{i=1}^{l} \mu_x^1(B_{\mathcal F}[b_i,\varepsilon_k]) \\
= \sum_{i=1}^{l} \mu_x^1(B_{\mathcal F}[a_i,1])) \cdot \mu_{a_i}^1(B_{\mathcal F}[a_i,\varepsilon_k]) +& \sum_{i=1}^{l} \mu_x^1(B_{\mathcal F}[b_i,1])) \cdot \mu_{b_i}^1(B_{\mathcal F}[b_i,\varepsilon_k]),
\end{align*}
using ~\eqref{eq:ai} and ~\eqref{eq:bi} we get,
\[ \mu_x^1(B_{\mathcal F}[x,1]) \geq \left( \sum_{i=1}^{l} \mu_x^1(B_{\mathcal F}[a_i,1])) + \sum_{i=1}^{l} \mu_x^1(B_{\mathcal F}[b_i,1])) \right) \cdot k \varepsilon_k.\]
Thus, from the conclusion of Lemma ~\ref{lemma:aux} we have that
\[ \mu_x^1(B_{\mathcal F}[x,1]) \geq 2 \left \lfloor \frac{1}{2}\left(\frac{1}{\varepsilon_k}-1\right)\right \rfloor \cdot \varepsilon_k \cdot k.\]
As the left side is finite and the right side goes to infinity as $k$ goes to infinity we obtain a contradiction. Thus indeed $\overline{\Delta}$ is finite.

%

\begin{lemma} \label{lemma:uniformDelta}
\[\overline{\Delta} = \underline{\Delta} = 1. \]
\end{lemma}
\begin{proof}
For a given $k \in \mathbb N^{*}$, we know that for any $x \in \overline{\Pi}$
\begin{equation}\label{eq:all}
\left| \frac{\mu^1_x(B_{\mathcal F}(x,\varepsilon_k))}{\varepsilon_k} - \overline{\Delta}  \right|   \leq \frac{1}{k} .\end{equation} 

Consider the closed ball $B=B_{\mathcal F}[x,1] \subset \mathcal F(x)$. Given $\epsilon > 0$ take $k_0 \in \mathbb N$ such that $k_{0}^{-1} < \epsilon$. Let $r=\varepsilon_k$ and let $a_i, b_i$ be as in Lemma \ref{lemma:aux}. Thus, we have a family of disjoint balls inside $B_{\mathcal F}[x,1]$ centered at the points $a_i$ and $b_i$, $1\leq i \leq l:=\left \lfloor \frac{1}{2}\left( \frac{1}{\varepsilon_k}-1\right) \right \rfloor$.
For each $1\leq i \leq l$ we have
\begin{equation}\label{eq:ai}
 \mu^1_{a_i}(B_{\mathcal F}[a_i,\varepsilon_k])) - \overline{\Delta}\varepsilon_k>- \epsilon \cdot \varepsilon_k \Rightarrow  \mu^1_{a_i}(B_{\mathcal F}[a_i,\varepsilon_k]) > \varepsilon_k(\overline{\Delta}-\epsilon),
\end{equation}
and similarly we obtain
\begin{equation}\label{eq:bi}
\mu^1_{b_i}(B_{\mathcal F}[b_i,\varepsilon_k]) > \varepsilon_k(\overline{\Delta}-\epsilon).\end{equation}
Therefore, by ~\eqref{eq:xai} and \eqref{eq:xbi}, 
\begin{eqnarray*}
 1=\mu^1_x(B_{\mathcal F}[x,1]) & > & \left( \sum_{i=1}^{l}\mu_{x}^{1}(B_{\mathcal F}[a_i,1]) + \sum_{i=1}^{l}\mu_{x}^{1}(B_{\mathcal F}[b_i,1]) \right) \cdot \varepsilon_k\cdot ({\overline{\Delta}-\epsilon}) .\\
\end{eqnarray*}
By Lemma \ref{lemma:aux} we have
\begin{eqnarray*}
 1=\mu^1_x(B_{\mathcal F}[x,1]) & > & 2 \cdot \left\lfloor \frac{1}{2}\left( \frac{1}{\varepsilon_k}-1\right) \right \rfloor \cdot \varepsilon_k\cdot ({\overline{\Delta}-\epsilon}),\\
\end{eqnarray*}
for every $k\geq 1$. Taking $k\rightarrow \infty$ we have that $$\overline{\Delta}\leq 1.$$
Similarly, by taking $r:=\varepsilon_k$ and  $a_i$, $b_i$, $1\leq i \leq l:=\left \lfloor \frac{1}{2}\left( \frac{1}{\varepsilon_k}-1\right) \right \rfloor$, as in Lemma \ref{lemma:aux2} we cover $B_{\mathcal F}[x,1]$ with $2l+2$ balls of radius $\varepsilon_k$. Now, we know that
\begin{align*}\mu^1_{a_i}(B_{\mathcal F}[a_i,\varepsilon_k])< & \epsilon\cdot \varepsilon_k + \overline{\Delta}\varepsilon_k \\
\mu^1_{b_i}(B_{\mathcal F}[b_i,\varepsilon_k])< & \epsilon\cdot \varepsilon_k + \overline{\Delta}\varepsilon_k,\end{align*}
for $1\leq i \leq \left \lfloor \frac{1}{2}\left( \frac{1}{\varepsilon_k}-1\right) \right \rfloor$.
Consequently, again using ~\eqref{eq:xai} and ~\eqref{eq:xbi}, we have
\[1= \mu^1_x(B_{\mathcal F}[x,1]) < \left( \sum_{i=1}^{l+1}\mu_{x}^{1}(B_{\mathcal F}[a_i,1]) + \sum_{i=1}^{l+1}\mu_{x}^{1}(B_{\mathcal F}[b_i,1]) \right) \cdot \varepsilon_k(\epsilon+\overline{\Delta}). \]
By Lemma \ref{lemma:aux2}
\[1= \mu^1_x(B_{\mathcal F}[x,1]) < (2l +2) \cdot \varepsilon_k(\epsilon+\overline{\Delta}) \Rightarrow \overline{\Delta}\leq 1.\]
Consequently we have that $\overline{\Delta} = 1$. Repeating the same argument with $\underline{\Delta}$ we conclude that 
\[\overline{\Delta} = \underline{\Delta} =1.\]
\end{proof}

Next we are able to conclude that $\mu^1_x$ is equivalent to the measure induced by the flow $\phi$ on the orbits 

\begin{lemma} \label{lemma:equallebesgue}
For almost every $x\in X$
\[\mu^1_x(B) = 2^{-1} \cdot \lambda_{\mathcal F(x)}(B)\]
where $\lambda_{\mathcal F(x)}$ is the measure on $\mathcal F(x)$ induced by the flow $\phi$ (i.e. $\lambda_{\mathcal F(x)}([x,y])=|t|$ if $y = \phi_t(x)$)
\end{lemma}
\begin{proof}
Take any typical plaque $\mathcal F(x)$ and any point $a\in \mathcal F(x)$. For each $r>0$ we can write the set $[a,\phi_r(a)]$ as a disjoint union as below
\[[a,\phi_r(a)] = \left(\bigcup_{j=0}^n [\phi_{2j\varepsilon_k}(a),\phi_{2(j+1)\varepsilon_k}(a)] \right) \cup J_k, \quad n:= \left \lfloor r/2\varepsilon_k \right \rfloor \]
where $J_k = [\phi_{2(n+1)\varepsilon_k}(a),\phi_r(a)]$. Each of the terms appearing on the right side of the previous equality, except for $J_k$, is a closed $\mathcal F$-ball of radius $\varepsilon_k$. By Lemma \ref{lemma:uniformDelta}, $\overline{\Delta} = \underline{\Delta}=1$ so
\[\varepsilon_k(1-1/k)<\mu_{c_j}^1([\phi_{2j\varepsilon_k}(a),\phi_{2(j+1)\varepsilon_k}(a)]) < \varepsilon_k(1+1/k),\]
where $c_j:=\phi_{2j\varepsilon_k}(a)+\varepsilon_k$, $j=0,1,\ldots, n$.
Also, we know that 
\[\mu_x^1(B_{\mathcal F}[c_j,\varepsilon_k]) = \mu_x^1(B_{\mathcal F}[c_j,1]) \cdot \mu_{c_j}^1(B_{\mathcal F}[c_j,\varepsilon_k]).\]
Therefore we have
\begin{align*}
\left( \sum_{j=0}^{n} \mu_x^1(B_{\mathcal F}[c_j,1]) \right) \cdot \varepsilon_k(1-1/k)  \leq &\quad \mu_x^1[a,\phi_r(a)] \leq   \\
\leq & \left( \sum_{j=0}^{n} \mu_x^1(B_{\mathcal F}[c_j,1]) \right)  \cdot \varepsilon_k(1+1/k) + \mu_x^1(J_k). \end{align*}
Repeating the argument of the proof of Lemma \ref{lemma:aux}, we see that $ \sum_{j=0}^{n} \mu_x^1(B_{\mathcal F}[c_j,1]) =n+1$. Thus
\[ \lfloor r/2\varepsilon_k \rfloor \cdot \varepsilon_k(1-1/k)  \leq  \mu_x^1([a,\phi_r(a)]) \leq \lfloor r/2\varepsilon_k \rfloor \cdot \varepsilon_k(1+1/k) + \mu_x^1(J_k).\]
Taking $k\rightarrow +\infty$ we have
\[\mu_x^1([a,\phi_r(a)]) = r/2\]
as we wanted to show.
%
%
%
%
\end{proof}

\textbf{Case 2: $D$ has null measure.}
Since $\overline \Pi_x$ is closed in the plaque $\mathcal F(x)$ for all $x \in D$ and $\mu(\overline \Pi)=1$, it is true that for a full measurable set $\mathfrak D$, if $ x\in \mathfrak D$ then $x\notin \overline{\Pi}$ if, and only if, there is $r>0$ with $\mu^1_x(B_{\mathcal F}(x,r)) = 0$. Now consider $\{q_1,q_2,...\}$ to be an enumeration of the rationals.
 
For each $i\geq 1$ let us define the function $S_i$ as
\[S_i(x) = \max\{q_j: 1\leq j \leq i \text{ and } \mu^1_y(B_{\mathcal F}(y,q_j))=0 \text{ for some } y \in \mathcal F(x) \}.\]

\begin{lemma}
 $S_i$ is an invariant measurable function for all $i \in \mathbb N$.
\end{lemma}
\begin{proof}
For each $i \in \mathbb N$ define the function $Q_i:  \mathfrak D \rightarrow [0,\infty)$ by
\[Q_i(x) = \mu^1_x(B_{\mathcal F}(x,q_i)).\]
By proposition \ref{prop:disinmeasurable2} $Q_i(x)$ is a measurable function for every $i$ and, by an standard measure theory argument, we may take a compact set $K \subset \mathfrak D$ of positive measure such that $Q_i | K$ is continuous for every $i$. Now, given $j \in \mathbb N$, let $\sigma$ be a permutation of $\{1,...,j\}$ such that $q_{\sigma(1)}<q_{\sigma(2)}<...<q_{\sigma(j)}$. 
Observe that for $$\mathfrak K = \bigcup_{n\in \mathbb Z} f^n(K)$$ we have
\[S_j^{-1}(\{q_{\sigma(j)}\}) \cap \mathfrak K = \bigcup_{n\in \mathbb Z} f^n ( \mathcal F(Q_{\sigma(j)}^{-1}(\{0\}) \cap K )),\]
which is a measurable set since $Q_{\sigma(n)}^{-1}(\{0\}) \cap K $ is a Borel set. Now,
\[S_j^{-1}(\{q_{\sigma(j-1)}\}) \cap \mathfrak K = \bigcup_{n\in \mathbb Z} f^n ( \mathcal F(Q_{\sigma(j-1)}^{-1}(\{0\}) \cap K )) \setminus (S_j^{-1}(\{q_{\sigma(j)}\}) \cap \mathfrak K),\]
which is also a measurable set. Inductively we prove that $S_j^{-1}(\{q_{\sigma(i)}\}) \cap \mathfrak K$ is measurable for all $1\leq i \leq j$. Since, by ergodicity, the set $\mathfrak K \subset \mathfrak D$ has full measure we conclude that $S_j(x)$ is measurable for every $j \geq 1$.
\end{proof}

Let $S(x) := \lim_{i\rightarrow \infty} S_i(x)$. $S$ is measurable and $f$-invariant thus it is constant almost everywhere, call this constant $r_0$. This means that for a full measure set $Y \subset \mathfrak D$, for every $x\in Y$ the plaque $\mathcal F(x)$ has a finite number of intervals of radius $r_0$ outside $\overline \Pi_x$. Let us call these intervals as ``bad" intervals.

Now consider the set $\mathfrak M$ formed by the median points of these ``bad" intervals of radius $r_0$. Notice that $\mathfrak M$ is a measurable set, since it is inside a set of zero measure and also that $f(\mathfrak M)=\mathfrak M$. 

Let $\varphi:(0,1) \times (0,1)^k \rightarrow U$ be a local chart for $\mathcal F$ such that the $\mathcal F(\mathfrak M \cap U)$, the $\mathcal F$ saturation of these ``bad" intervals inside $U$, has positive measure. Set $\Sigma := \pi_1(\varphi^{-1}(\mathfrak M\cap U))$, where $\pi_1:(0,1) \times (0,1)^k \rightarrow (0,1)$ is the projection onto the first coordinate. Now we may apply The Measurable Choice Theorem \ref{theo:MCT} to obtain a measurable function $ \mathfrak F: \Sigma \rightarrow (0,1)$ such that $(x,\mathfrak F(x)) \in  \varphi^{-1}(\mathfrak M\cap U)$ for all $x \in \Sigma$. Again by standard arguments, using Lusin's theorem, we may assume $\Sigma$ to be compact and such that $\mathfrak F$ is a continuous function.

Now consider the set $\mathfrak M_0:= \varphi(\text{graph }\mathfrak F)$, which is a Borel set since the graph of $\mathfrak F$ is a compact set. Notice that our construction implies that $\mathcal F(\mathfrak M_0)$ has positive measure. Now define the following $f$ invariant set $$\mathfrak M_1:= \bigcup_{n \in \mathbb Z} f^n(\mathfrak M_0).$$ By ergodicity the set $\mathcal F (\mathfrak M_1)$ has full measure.


The set $\mathfrak M_1$ intersects almost each plaque in a finite (constant) number of points. Notice that, for each $r \in \mathbb R_+$ the invariant set $$\mathfrak M_1^r:= \bigcup_{x \in \mathfrak M_1} B_{\mathcal F}(x,r)$$ has zero or full measure. Let $\alpha_0$ such that $\mu(\mathfrak M_1^r)=0$ if $r < \alpha_0$ and $\mu(\mathfrak M_1^r)=1$ if $r \geq \alpha_0$. This implies that the extreme points of $B_{\mathcal F}(x,\alpha_0)$ for $x \in \mathfrak M_1$ forms a set of atoms.

Which is an absurd, because we are assuming we are in the non-atomic case. The case $\overline{\Delta}=\infty$ is similar. 

The measure $\mu_{\mathcal F(x)}$ of the statement of the result is given, due to Lemma \ref{lemma:equallebesgue}, as $2^{-1}\lambda_{\mathcal F(x)}$.

Let us now work with the case $\mu(Per(\phi))>0$. By ergodicity of $f$ and the $f$ invariance of $Per(\phi)$ we have $\mu(Per(\phi))=1$. Hence the partition of $X$ given by each periodic orbit of $\phi$ and the set $X\setminus Per(\phi)$ forms a measurable partition (e.g. \cite[Proposition 2.5]{PTV}). We can then disintegrate $\mu$ on this partition. Denote the family of conditional measures by $\{\mu_{\mathcal F(x)}\}$. If the set of singularities for the flow $\phi$ have positive measure then it is clear that one should have full measure, in particular the measure is atomic and we fall on the first item. If not then we repeat the prove but instead of working with $\mu_x^1$ we can simply work with the disintegrated measures $\mu_{\mathcal F(x)}$ and the theorem follows.

$\hfill \square$

\section*{Acknowledgements}

We would like to thank Ali Tahzibi for usefull conversations. G.P. was partially supported by FAPESP grant 2016/05384-0 and R.V. was partially supported by FAPESP grant 2016/22475-9.
\bibliographystyle{plain}
\bibliography{Referencias.bib}
\end{document}